\documentclass[a4paper,leqno]{amsart}

\usepackage{diagrams}

\usepackage{amssymb,amsthm,amsmath,amsrefs}

\usepackage{ifthen}

\newcounter{mylisti} \newcounter{mylistii}
\newcounter{nest}
\newcommand{\defaultlabel}{}

\newenvironment{mylist}[1]{%
  \addtocounter{nest}{1}
  \ifthenelse{\value{nest}=1}{%
    \renewcommand{\defaultlabel}{(\roman{mylisti})\hfill}}{%
    \renewcommand{\defaultlabel}{(\alph{mylistii})\hfill}}
  \begin{list}{\defaultlabel}{%
      \ifthenelse{\value{nest}=1}{\usecounter{mylisti}}{%
        \usecounter{mylistii}}
      
      \addtolength{\itemsep}{0.5ex}
      \settowidth{\labelwidth}{#1}
      \setlength{\leftmargin}{\labelwidth}
      \addtolength{\leftmargin}{\labelsep}}}{\addtocounter{nest}{-1}
\end{list}}

\newcommand{\be}{\ensuremath{\mathbb E}}

\newcommand{\bn}{\ensuremath{\mathbb N}}

\newcommand{\br}{\ensuremath{\mathbb R}}

\newcommand{\cC}{\ensuremath{\mathcal C}}

\newcommand{\cJ}{\ensuremath{\mathcal J}}
\newcommand{\cK}{\ensuremath{\mathcal K}}
\newcommand{\cL}{\ensuremath{\mathcal L}}

\newcommand{\cS}{\ensuremath{\mathcal S}}

\newcommand{\yt}{\ensuremath{\tilde{y}}}
\newcommand{\zt}{\ensuremath{\tilde{z}}}

\newcommand{\At}{\ensuremath{\tilde{A}}}

\newcommand{\vv}{\ensuremath{\boldsymbol{v}}}
\newcommand{\vw}{\ensuremath{\boldsymbol{w}}}

\newcommand{\abs}[1]{\lvert #1\rvert}
\newcommand{\bigabs}[1]{\big\lvert #1\big\rvert}

\newcommand{\biggabs}[1]{\bigg\lvert #1\bigg\rvert}
\newcommand{\sabs}[1]{\left\lvert #1\right\rvert}

\newcommand{\diag}{\operatorname{diag}}

\newcommand{\fss}{\ensuremath{\mathcal F\mathcal S}}

\newcommand{\intp}[1]{\ensuremath{\lfloor #1\rfloor}}

\newcommand{\Bigintp}[1]{\ensuremath{\Big\lfloor #1\Big\rfloor}}

\newcommand{\ip}[2]{\ensuremath{\langle #1,#2\rangle}}
\newcommand{\bigip}[2]{\ensuremath{\big\langle #1,#2\big\rangle}}
\newcommand{\Bigip}[2]{\ensuremath{\Big\langle #1,#2\Big\rangle}}

\newcommand{\join}{\vee}

\newcommand{\meet}{\wedge}

\newcommand{\norm}[1]{\lVert #1\rVert}
\newcommand{\bignorm}[1]{\big\lVert #1\big\rVert}

\newcommand{\biggnorm}[1]{\bigg\lVert #1\bigg\rVert}
\newcommand{\snorm}[1]{\left\lVert #1\right\rVert}

\newcommand{\sgn}{\ensuremath{\mathrm{sign}}}

\newcommand{\spn}{\ensuremath{\mathrm{span}}}

\newcommand{\co}{\mathrm{c}_0}
\newcommand{\coo}{\mathrm{c}_{00}}

\newcommand{\continuum}{\mathfrak{c}}
\newcommand{\md}{\mathrm{d}}

\newcommand{\vare}{\varepsilon}
\newcommand{\varf}{\varphi}

\newcommand{\ds}{\displaystyle}

\newcommand{\ts}{\textstyle}

\newcommand{\eg}{\textit{e.g.,}\ }

\newcommand{\ie}{\textit{i.e.,}\ }

\newtheorem{thm}{Theorem}
\newtheorem{mainthm}{Theorem}

\newtheorem*{mainproblem*}{Problem}

\newtheorem{lem}[thm]{Lemma}
\newtheorem{prop}[thm]{Proposition}
\newtheorem{cor}[thm]{Corollary}
\newtheorem{problem}[thm]{Problem}

\theoremstyle{definition}

\theoremstyle{remark}

\newtheorem*{rem}{Remark}

\newtheorem*{ex}{Example}

\begin{document}

\allowdisplaybreaks

\title[$\cL(\ell_p\oplus \ell_q)$ has infinitely many closed
  ideals]{The algebra of bounded linear operators on $\ell_p\oplus
  \ell_q$ has infinitely many closed ideals}

\author{Th.~Schlumprecht}
\address{Department of Mathematics, Texas A\&M University, College
  Station, TX 77843, USA and Faculty of Electrical Engineering, Czech
  Technical University in Prague,  Zikova 4, 166 27, Prague}
\email{schlump@math.tamu.edu}

\author{A.~Zs\'ak}
\address{Peterhouse, Cambridge, CB2 1RD, UK}
\email{a.zsak@dpmms.cam.ac.uk}

\date{10 September 2014}

\thanks{The first author's research was supported by NSF grant
  DMS1160633. The second author was supported by the 2014 Workshop in
  Analysis and Probability at Texas A\&M University.}
\subjclass[2010]{}

\begin{abstract}
  We prove that in the reflexive range $1<p<q<\infty$, the algebra
  $\cL(\ell_p\oplus\ell_q)$ of all bounded linear operators on
  $\ell_p\oplus\ell_q$ has infinitely many closed ideals. This solves
  a problem raised by A.~Pietsch~\cite{pietsch:78}*{Problem 5.3.3}
  in his book, `Operator ideals'.
\end{abstract}

\maketitle

\section{Introduction}

Classifying the closed ideals of the algebra $\cL(\ell_p\oplus\ell_q)$
of bounded linear operators on $\ell_p\oplus\ell_q$ has a long
history. There were several results proved in the 1970s, and the
reader is referred to the book by Pietsch~\cite{pietsch:78}*{Chapter
  5} for details. In particular,~\cite{pietsch:78}*{Theorem 5.3.2}
asserts that for $1\leq p<q<\infty$ there are exactly two maximal
ideals of $\cL(\ell_p\oplus\ell_q)$. These are the closures of the ideals
of operators factoring through $\ell_p$ and $\ell_q$, respectively.
\cite{pietsch:78}*{Theorem 5.3.2} also establishes a one-to-one
correspondence between the set of all other closed, proper ideals of
$\cL(\ell_p\oplus\ell_q)$ and the set of all closed  ideals of
$\cL(\ell_p,\ell_q)$. Here an ideal of $\cL(\ell_p,\ell_q)$ is a
subspace $\cJ$ of $\cL(\ell_p,\ell_q)$ with the property that
$ATB\in\cJ$ whenever $A\in\cL(\ell_q), T\in\cJ$ and
$B\in\cL(\ell_p)$. Pietsch raises the following
problem.
\begin{mainproblem*}[\cite{pietsch:78}*{Problem 5.3.3}]
  For $1\leq p<q$ does $\cL(\ell_p,\ell_q)$ have infinitely many
  closed ideals?
\end{mainproblem*}
Since we are dealing with Banach spaces with bases, it is clear that
the compact operators $\cK=\cK(\ell_p,\ell_q)$ are the smallest
non-trivial (\ie non-zero) closed ideal. Since the formal inclusion
map $I_{p,q}\colon \ell_p\to\ell_q$ is not compact, $\cK$ is a proper
ideal. For anyone well versed in basis techniques, it is a not too
difficult exercise that every operator in $\cL(\ell_p,\ell_q)$ is
strictly singular, and that every non-compact operator factors
$I_{p,q}$. It follows that $\{0\}\subsetneq \cK\subsetneq
\cJ^{I_{p,q}} \subset \cS=\cL(\ell_p,\ell_q)$, where $\cJ^{I_{p,q}}$
is the closure of the ideal of operators factoring through $I_{p,q}$,
$\cS=\cS(\ell_p,\ell_q)$ is the ideal of strictly singular
operators, and moreover any other closed ideal of $\cL(\ell_p,\ell_q)$
must contain $\cJ^{I_{p,q}}$. It is, however, not obvious that
$\cJ^{I_{p,q}}$ is proper. This was shown for $1<p<q<\infty$ by
V.~D.~Milman~\cite{milman:70} who first proved that $I_{p,q}$ is
finitely strictly singular, and then exhibited an operator in
$\cL(\ell_p,\ell_q)$ that is not finitely strictly
singular. (Definitions will be given in Section~\ref{sec:definitions}
below.) The next significant result was proved by B.~Sari,
N.~Tomczak-Jaegermann, V.G.~Troitsky and the first named
author. In~\cite{sari-schlump-tomczak-troitsky:07} they showed that
for $1<p<2<q<\infty$, the
ideal $\fss$ of finitely strictly singular operators and the ideal
$\cJ^{\ell_2}$ generated by operators factoring through $\ell_2$ are
proper, distinct, and distinct from the ones above. So in this case
$\cL(\ell_p.\ell_q)$ has at least four non-trivial, proper closed
ideals. Later the first named author~\cite{schlump:12} found two more
ideals, again in the range $1<p<2<q<\infty$, namely the ideals
generated by operators factoring through the formal inclusion maps
$I_{p,2}$ and $I_{2,q}$, respectively. These new ideals lie between
$I_{p,q}$ and $\fss\cap\cJ^{\ell_2}$, and hence the fact they are
incomparable shows
that $\cJ^{I_{p,q}}\neq\fss\cap\cJ^{\ell_2}$, and thus there are at
least seven non-trivial, proper, closed ideals in $\cL(\ell_p,\ell_q)$
when $1<p<2<q<\infty$.

The main result of this paper is a solution of Pietsch's question in
the reflexive range.
\begin{mainthm}
  \label{mainthm:infinite}
  For all $1<p<q<\infty$ there is a chain of size the continuum
  consisting of closed ideals in $\cL(\ell_p,\ell_q)$ that lie between
  the ideals $\cJ^{I_{p,q}}$ and $\fss$.
\end{mainthm}

Moreover, we obtain the following refinement.

\begin{mainthm}
  \label{mainthm:refinement}
  For all $1<p<2<q<\infty$ there is a chain of size the continuum
  consisting of closed ideals in $\cL(\ell_p,\ell_q)$ that lie between
  $\cJ^{I_{p,q}}$ and $\cJ^{I_{2,q}}$.
\end{mainthm}

We note that these results provide further examples of separable
reflexive Banach spaces $X$ such that $\cL(X)$ has continuum many
closed ideals. The first such examples were given by
A.~Pietsch~\cite{pietsch:02}.

The paper is organized as follows. In Section~\ref{sec:definitions} we
introduce definitions, notations and certain complemented subspaces of
$\ell_p$ that will later lead to new ideals. A crucial r\^ole is
played here by H.~P.~Rosenthal's famous $X_p$ spaces, which we recall
in some detail. We shall use $\ell_p$-direct sums of
finite-dimensional versions of $X_p$. The so called lower fundamental
function (defined below) of these direct sums will be computed at the
end of Section~\ref{sec:definitions}. In Section~\ref{sec:key-lem}
we prove a key lemma that will be at the heart of the proof of our main
results. The latter will be presented in
Section~\ref{sec:main-results}. We conclude with a list of open
problems in Section~\ref{sec:problems}.

\section{Definitions and known results}
\label{sec:definitions}

In this section we introduce a large number of definitions as well as
some preliminary results. This long section is divided into digestable
chunks.

\subsection{$\ell_p$ spaces}
\label{subsec:lp-spaces}

For $1\leq p\leq \infty$ we denote by $p'$ the conjugate index of
$p$.  So we have $\frac{1}{p}+\frac{1}{p'}=1$. We will denote by
$\{e_{p,i}:\,1\leq i<\infty\}$ the unit vector basis of $\ell_p$ when
$1\leq p<\infty$ and of $\co$ when $p=\infty$. For $n\in\bn$, the unit
vector basis of $\ell_p^n$ will be denoted by
$\big\{e^{(n)}_{p,j}:\,1\leq j\leq n\big\}$.

Fix $p\in(1,\infty)$ and define $Z_p$ to be the $\ell_p$-direct sum
$Z_p=\big( \bigoplus _{n=1}^\infty \ell_2^n \big)_{\ell_p}$. This has
canonical unit vector basis $\big\{ e^{(n)}_{2,j}:\,n\in\bn,\ 1\leq
j\leq n\big\}$. By Kintchine's inequality, the spaces $\ell_2^n$,
$n\in\bn$, are uniformly complemented in $\ell_p$. It follows that
$Z_p$ is also complemented in $\ell_p$, and hence isomorphic to it by
Pe{\l}czy\'nski's Decomposition Theorem. We fix once and for all an isomorphism
$U_p\colon Z_p\to\ell_p$. Although the spaces $Z_p$ and $\ell_p$ are
isomorphic, their canonical unit vector bases $\big\{
e^{(n)}_{2,j}:\,n\in\bn,\ 1\leq j\leq n\big\}$ and $\big\{e_{p,i}:\,
1\leq i<\infty \big\}$, respectively, are very different when $p\neq
2$. This is an example of the following convention that we adopt
throughout this paper. For us a Banach space means a Banach space
together with a fixed basis (which will always be normalized and
1-unconditional). We will use different notation for the same space if
we consider more than one basis.

\subsection{Operator ideals}
\label{subsec:op-ideal}

Let $X$ and $Y$ be Banach spaces. We denote by $\cL(X,Y)$ the space of
all bounded, linear operators from $X$ to $Y$. We write $\cL(X)$ for
$\cL(X,X)$, and we let $I_X\in\cL(X)$ stand for the identity operator
on $X$. If $X$ and $Y$ have fixed bases $(x_i)$ and $(y_i)$,
respectively, such that $(y_i)$ is dominated by $(x_i)$, then we write
$I_{X,Y}$ for the formal inclusion $X\to Y$ defined by
\[
I_{X,Y} \Big( \sum_i a_i x_i \Big) =\sum _i a_i y_i\ ,
\]
which is well defined and bounded. When $X$ (respectively, $Y$) is
$\ell_p$, $1\leq p<\infty$, then we write $I_{p,Y}$ (respectively,
$I_{X,p}$) instead of $I_{X,Y}$. Similarly, we write $I_{\infty,Y}$
instead of $I_{\co,Y}$, etc.

By an \emph{ideal of $\cL(X,Y)$ }we mean a subspace $\cJ$ of
$\cL(X,Y)$ satisfying the ideal property: $ATB\in\cJ$ for all
$A\in\cL(Y), T\in\cJ$ and $B\in\cL(X)$. When $X=Y$ this coincides with
usual algebraic notion of an ideal of the algebra $\cL(X)$. A closed
ideal is an ideal that is closed in the operator norm. Clearly, the
closure of an ideal is a closed ideal.

Our notion of ideal is equivalent to Pietsch's notion of \emph{operator
ideal }\cite{pietsch:78}. The latter is a functor that assigns to each
pair $(V,W)$ of Banach spaces a subspace $\cJ(V,W)$ of $\cL(V,W)$ such
that for all Banach spaces $U, V, W, Z$ and for all $A\in\cL(W,Z),
T\in\cJ(V,W), B\in \cL(U,V)$, we have $ATB\in\cJ(U,Z)$. This is called
a \emph{closed operator ideal }if $\cJ(V,W)$ is a closed subspace of
$\cL(V,W)$ for all $V,W$. Given a (closed) operator ideal $\cJ$, it is
clear that $\cJ(V,W)$ is then an ideal (respectively, closed ideal),
in the above sense, of $\cL(V,W)$ for all  spaces $V,W$. Conversely,
given a (closed) ideal $\cJ$ of $\cL(X,Y)$ in the above sense,
the functor that assigns to $(V,W)$ (the closure of) the set of all
finite sums of operators of the form $ATB$ with $A\in\cL(Y,W),
T\in\cJ$ and $B\in\cL(V,X)$ is a (closed) operator ideal in the sense
of Pietsch for which $\cJ(X,Y)=\cJ$.

In this paper we shall only deal with closed ideals. Recall that
$T\in\cL(X,Y)$ is \emph{strictly singular }if it not an isomorphism on
any infinite-dimensional subspace of $X$, and $T$ is \emph{finitely
  strictly singular }if for all $\vare>0$ there exists $n\in\bn$ such
that for every subspace $E$ of $X$ with $\dim E\geq n$ there exists
$x\in E$ such that $\norm{Tx}<\vare\norm{x}$. We denote by
$\cK(X,Y)\subset \fss(X,Y)\subset \cS(X,Y)$ the ideals of,
respectively, compact, finitely strictly singular and strictly singular
operators. When $X=Y$ these become $\cK(X)\subset \fss(X)\subset
\cS(X)$. It is not hard to see that these are all closed operator
ideals.

For an operator $T\colon U\to V$ between Banach spaces $U$ and $V$, we
let $\cJ^T=\cJ^T(X,Y)$ be the closed ideal of $\cL(X,Y)$ generated by
operators factoring through $T$. Thus $\cJ^T$ is the closure in
$\cL(X,Y)$ of all finite sums of operators of the form $ATB$, where
$A\in \cL(V,Y)$ and $B\in\ \cL(X,U)$. When $U=V$ and $T=I_U$ then we
write $\cJ^U$ instead of $\cJ^{I_U}$.

\subsection{Current state of affairs}
\label{subsec:current}

We recall that for $1\leq p<q<\infty$ the algebra
$\cL(\ell_p\oplus\ell_q)$ has exactly two
maximal ideals, the closures of the ideals of operators factoring
through $\ell_p$ and $\ell_q$, respectively, and the set of all
non-maximal closed ideals of $\cL(\ell_p\oplus\ell_q)$ is in a
one-to-one, inclusion-preserving correspondence with the set of
proper, closed ideals of $\cL(\ell_p,\ell_q)$. These results are
stated in~\cite{pietsch:78}*{Theorem 5.3.2}. We also refer the reader
to Section~2 of~\cite{schlump:12}.

Using the notation established in~\ref{subsec:op-ideal}, the following
diagrams summarize what we know about non-trivial closed ideals of
$\cL(\ell_p,\ell_q)$. When $1 <p<q<\infty$ we have
\[
\cK \subsetneq \cJ^{I_{p,q}} \subset \fss \subsetneq
\cL(\ell_p,\ell_q)\ .
\]
Recall that the last two inclusions are due to
V.~D.~Milman~\cite{milman:70}. Also, if $\cJ$ is a proper, closed
ideal distinct from $\cK$, then $\cJ$ contains $\cJ^{I_{p,q}}$.

Combining the main results of~\cite{sari-schlump-tomczak-troitsky:07}
and~\cite{schlump:12}, we obtain the following picture in the range
$1<p<2<q<\infty$.
\[
\begin{diagram}[w=2.5em,h=3ex]
  &&&& \boxed{\cJ^{I_{p,2}}} &&&& \boxed{\fss} &&&& \\
  &&& \ruTo && \rdTo && \ruTo && \rdTo &&& \\
  \boxed{\cK} & \rTo & \boxed{\cJ^{I_{p,q}}} &&
  \nparallel && \boxed{\fss\cap\cJ^{\ell_2}} &&
  \nparallel && \boxed{\fss\join \cJ^{\ell_2}} & \rTo &
  \boxed{\cL(\ell_p,\ell_q)} \\
  &&& \rdTo && \ruTo && \rdTo && \ruTo &&& \\
  &&&& \boxed{\cJ^{I_{2,q}}} &&&& \boxed{\cJ^{\ell_2}} &&&&
\end{diagram}
\]
Here arrows stand for inclusion. It is not known whether the ideal
$\fss\join \cJ^{\ell_2}$, the smallest closed ideal containing $\fss$ and
$\cJ^{\ell_2}$, is proper. All other inclusions are strict. The ideals
$\cK$ and $\cJ^{I_{p,q}}$ are the smallest, respectively second
smallest ideals. In~\cite{sari-schlump-tomczak-troitsky:07} it was
also shown that any ideal containing an operator not in $\fss$ must
contain $\cJ^{\ell_2}$. It follows that there is no ideal between
$\fss\cap \cJ^{\ell_2}$ and $\cJ^{\ell_2}$, and that $\fss\join
\cJ^{\ell_2}$ is the only immediate successor of $\fss$. Two ideals
connected by $\nparallel$ are incomparable. It is not known whether
$\cJ^{I_{p,q}}$ is a proper subset of $\cJ^{I_{p,2}}\cap\cJ^{I_{2,q}}$.

\subsection{Rosenthal's $X_p$ spaces}
\label{subsec:rosenthal-xp}

In 1970 H.~P.~Rosenthal published an influential
paper~\cite{rosenthal:70a} with important consequences both for the
theory of $\cL_p$ spaces, and for probablity theory. This paper grew
out from his study of sequences of independent random variables with
mean zero in $L_p=L_p[0,1]$ for $1<p<\infty$. It led to the
introduction of the spaces $X_p$ which we now describe. Let
$2<p<\infty$ and $\vw=(w_n)_{n=1}^\infty$ be a sequence in
$(0,1]$. The space $X_{p,\vw}$ is the completion of the space $\coo$
  of finite scalar sequences with respect to the norm
\[
\norm{(a_n)}_{p,\vw}=\Big(\sum \abs{a_n}^p \Big)^{\frac1{p}} \join \Big( \sum
w_n^2 \abs{a_n}^2 \Big)^{\frac12}\ .
\]
Rosenthal proved the following~\cite{rosenthal:70a}*{Theorem 4}. Let
$(f_n)$ be a sequence of independent, symmetric, 3-valued random
variables, and let $Y_p$ be the closed linear span of $(f_n)$ in $L_p$
for $1<p<\infty$. Then $Y_p$ is $K_p$-complemented in $L_p$, where
$K_p$ is a constant depending only on $p$. Moreover, if $p>2$, then
$Y_p$ is isomorphic to $X_{p,\vw}$, where
$w_n=\norm{f_n}_{L_2}/\norm{f_n}_{L_p}$, and $Y_{p'}$ is isomorphic to
$X_{p,\vw}^*$. More precisely, the proof shows that if $1<p<\infty$
and we assume that $\norm{f_n}_{L_p}=1$ for all $n\in\bn$ (as we
clearly may), then there is a projection $P_p\colon L_p\to L_p$ onto
$Y_p$ given by $P_pf=\sum a_n f_n$, where $a_n=\int_0^1 f(x)f_n(x)\,\md
x\cdot \norm{f_n}_{L_2}^{-2}$ for all $n\in\bn$, which satisfies
\begin{align}
  \label{eq:xp-isom-p>2}
  \norm{(a_n)}_{p,\vw} &\leq \norm{P_pf}_{L_p}
  \leq K_p\norm{(a_n)}_{p,\vw} \leq K_p\norm{f}_{L_p} &&\text{if
  }2<p<\infty\ ,\\[2ex]
  \label{eq:xp-isom-p<2}
  \frac1{K_p}\norm{(a_n)}^*_{p',\vw} &\leq
  \norm{P_pf}_{L_p} \leq
  \norm{(a_n)}^*_{p',\vw}\leq K_p\norm{f}_{L_p}&&\text{if }1<p<2\ .
\end{align}
Here in the case $2<p<\infty$ we have $w_n=\norm{f_n}_{L_2}$ for all
$n\in\bn$, whereas if $1<p<2$, then $w_n=\norm{f_n}_{L_2}^{-1}$, and
$\norm{(a_n)}^*_{p',\vw}$ denotes the norm of $\sum a_n e_n^*$ in the
dual space $X_{p',\vw}^*$, and $(e_n^*)$ is the sequence biorthogonal
to the unit vector basis $(e_n)$ of $X_{p',\vw}$. Note that $P_2$ is
simply the orthogonal projection of $L_2$ onto $Y_2$ and $K_2=1$; for
$2<p<\infty$ we obtain $P_p$ as the restriction to $L_p$ of $P_2$, and
for $1<p<2$ we have $P_p=P_{p'}^*$ and $K_p=K_{p'}$. It follows
from~\cite{rosenthal:70a}*{Lemma 2} that in all cases the sequence
$(f_n)$ is a 1-unconditional basis of $Y_p$.

Let $2<p<\infty$. It is easy to see that $X_{p,\vw}$ is isomorphic to
one of the spaces $\ell_2, \ell_p$ and $\ell_2\oplus\ell_p$ unless
$(w_n)$ satisfies
\begin{equation}
  \label{eq:xp-condition}
  \liminf w_n=0\quad\text{and}\quad \sum_{n,\ w_n<\vare}
  w_n^{\frac{2p}{p-2}}=\infty \ \text{for all }\vare>0\ .
\end{equation}
Rosenthal proved that if sequences $(w_n)$ and $(w_n')$ both
satisfy~\eqref{eq:xp-condition}, then the corresponding spaces
$X_{p,\vw}$ and $X_{p,\vw'}$ are isomorphic and distinct from any of
the spaces $\ell_2, \ell_p$ and $\ell_2\oplus\ell_p$.

In this paper we shall use finite-dimensional versions of Rosenthal's
$X_p$ spaces, and we will only need the result about the existence of
well-isomorphic and well-complemented copies in $L_p$. We begin with
some definitions.

Given $2<p<\infty,\ 0<w\leq 1$ and $n\in\bn$, we denote by
$E^{(n)}_{p,w}$ the Banach space
$\big(\br^n,\norm{\cdot}_{p,w}\big)$, where
\[
\bignorm{(a_j)_{j=1}^n}_{p,w}=\left( \sum_{j=1}^n
\abs{a_j}^{p}\right)^{\frac{1}{p}} \join w\left( \sum_{j=1}^n
\abs{a_j}^2\right)^{\frac12}\ .
\]
We write $\big\{ e^{(n)}_j:\,1\leq j\leq n\big\}$ for the unit
vector basis of $E^{(n)}_{p,w}$, and we denote by $\big\{
e^{(n)*}_j:\,1\leq j\leq n\big\}$ the unit vector basis of the dual
space $\big(E^{(n)}_{p,w}\big)^*$, which is biorthogonal to the unit
vector basis of $E^{(n)}_{p,w}$.

Given $1<p<2,\ 0<w\leq 1$ and $n\in\bn$,  we fix once and for all a
sequence $f^{(n)}_j=f^{(n)}_{p,w,j}$, $1\leq j\leq n$, of independent
symmetric, 3-valued random variables with $\norm{f^{(n)}_j}_{L_p}=1$
and $\norm{f^{(n)}_j}_{L_2}=\frac1{w}$ for $1\leq j\leq n$. We then
define $F^{(n)}_{p,w}$ to be the subspace $\spn\big\{ f^{(n)}_j:\,1\leq
j\leq n\big\}$ of $L_p$. It follows from~\eqref{eq:xp-isom-p<2} that
\begin{equation}
  \label{eq:xp-isom-p<2-n}
  \frac1{K_p}\snorm{\sum_{j=1}^n a_je^{(n)*}_j} \leq
  \snorm{\sum_{j=1}^n a_j f^{(n)}_j}_{L_p} \leq
  \snorm{\sum_{j=1}^n a_je^{(n)*}_j}\ ,
\end{equation}
where $\big\{ e^{(n)*}_j:\,1\leq j\leq n\big\}$ is the unit vector basis
of the dual space $\big(E^{(n)}_{p',w}\big)^*$ as defined above.
Since the $f^{(n)}_j$ are 3-valued, $F^{(n)}_{p,w}$ is a subspace of
the span of indicator functions of $3^n$ pairwise disjoint sets. Thus,
we can and will think of $F^{(n)}_{p,w}$ as a subspace of
$\ell_p^{k_n}$, where $k_n=3^n$.
\begin{prop}
  \label{prop:properties-of-Fn}
  Let $1<p<2,\ 0<w\leq 1$ and $n\in\bn$. Then
  \begin{mylist}{(iii)}
  \item
    $\big\{f^{(n)}_{j}:\,1\leq j\leq n\big\}$ is a normalized,
    1-unconditional basis of $F^{(n)}_{p,w}$.
  \item
    There exists a projection
    $P^{(n)}_{p,w}\colon\ell_p^{k_n}\to\ell_p^{k_n}$ onto
    $F^{(n)}_{p,w}$ with $\bignorm{P^{(n)}_{p,w}}\leq K_p$.
  \item
    For each $1\leq k\leq n$ and for every $A\subset \{1,\dots,n\}$
    with $\abs{A}=k$ we have
    \[
    \frac1{K_p}\cdot \Big(k^{\frac{1}{p}} \meet \tfrac{1}{w}
    k^{\frac12}\Big)
    \leq\biggnorm{\sum_{j\in A} f^{(n)}_j} \leq k^{\frac{1}{p}} \meet
    \tfrac{1}{w} k^{\frac12}\ .
    \]
  \end{mylist}
\end{prop}
\begin{proof}
  (i) and (ii) follow from the results of H.~P.~Rosenthal,
  \cite{rosenthal:70a}*{Theorem 4} and~\cite{rosenthal:70a}*{Lemma 2},
  that we cited above. By~\eqref{eq:xp-isom-p<2-n} we will have
  proved~(iii) if we show that
  \[
  \snorm{\sum_{j=1}^k e^{(n)*}_j} =k^{\frac{1}{p}} \meet \tfrac{1}{w}
  k^{\frac12}\ ,
  \]
  where $\big\{e^{(n)*}_j:\,1\leq j\leq n\big\}$ is the unit vector
  basis of $\big(E^{(n)}_{p',w}\big)^*$ as defined above. Now, by
  definition, we have
  \[
  \snorm{\sum_{j=1}^k e^{(n)*}_j} = \max \left\{
  \sum_{j=1}^k a_j:\, \sum_{j=1}^k \abs{a_j}^{p'} \leq 1 \text{ and }
  w^2 \sum_{j=1}^k \abs{a_j}^2 \leq 1 \right\}\ .
  \]
  Then by symmetry of $\norm{\cdot}_{p',w}$, the maximum occurs when
  $a_1=a_2=\dots = a_k=t$, say. So
  \[
  \snorm{\sum_{j=1}^k e^{(n)*}_j} = \max \left\{ kt:\,
  kt^{p'}\leq 1 \text{ and } w^2kt^2\leq 1 \right\} = k^{\frac{1}{p}}
  \meet \tfrac{1}{w} k^{\frac12}  \ ,
  \]
  as claimed.
\end{proof}

\begin{rem}
  We mention two extreme examples. When $w=1$, then $E^{(n)}_{p',w}\cong
  \ell^n_2$, and when $w=n^{-\eta}$ with $\eta=\frac{1}{p}-\frac12$,
  then $E^{(n)}_{p',w}\cong \ell^n_{p'}$. In both cases the formal
  identity map is an isometric isomorphism. It follows
  by~\eqref{eq:xp-isom-p<2-n} that if $w=1$, then $F^{(n)}_{p,w}\sim
  \ell_2^n$, and if $w=n^{-\eta}$, then
  $F^{(n)}_{p,w}\sim\ell_p^n$. In both cases the formal identity is a
  $K_p$-isomorphism.
\end{rem}

\subsection{The spaces $Y_{p,\vv}$}
\label{subsec:the-space-Yv}

Fix $1<p<2$, and let $\vv=(v_n)$ be a decreasing sequence in
$(0,1]$. For each $n\in\bn$, let $F_n$ be the subspace
$F^{(n)}_{p,v_n}$ of $\ell_p^{k_n}$ with
basis $\big\{f^{(n)}_j:\,1\leq j\leq n\big\}$ as
defined in Section~\ref{subsec:rosenthal-xp} above. We introduce the
space $Y_{p,\vv}$ defined to be the $\ell_p$-direct sum
$Y_{p,\vv}=\big( \bigoplus _{n=1}^\infty F_n \big)_{\ell_p}$. This is
a $K_p$-complemented subspace of $\ell_p$. Indeed, the diagonal
operator
\[
P_{p,\vv} = \diag\big(P^{(n)}_{p,v_n}\big)
\colon \ell_p\cong \left( \bigoplus _{n=1}^\infty \ell_p^{k_n}
\right)_{\ell_p} \to  \ell_p\cong \left( \bigoplus _{n=1}^\infty
\ell_p^{k_n} \right)_{\ell_p}
\]
is a projection onto $Y_{p,\vv}$,
where $P^{(n)}_{p,v_n}$ is the projection given by
Proposition~\ref{prop:properties-of-Fn}(ii). Furthermore, $Y_{p,\vv}$
is equipped with the normalized, 1-unconditional basis $\big\{
f^{(n)}_j:\,n\in\bn,\ 1\leq j\leq n\big\}$. Note that $Y_{p,\vv}$, as a
complemented subspace of $\ell_p$, is isomorphic to $\ell_p$. However,
we shall never make this identification, and instead consider
$Y_{p,\vv}$ as a complemented subspace of $\ell_p$ with corresponding
projection $P_{p,\vv}$ fixed as above.

We conclude this section by proving a norm estimate,
Lemma~\ref{lem:lower-fund-fn-of-Yv} below, on sums of basis vectors of
$Y_{p,\vv}$. We begin with fixing some notation.
Let $X$ be a Banach space with a fixed basis $(x_i)$ (finite or
infinite). Let $N=\bn$ if $\dim(X)=\infty$, and
$N=\{ 1,2,\dots,\dim(X)\}$ otherwise. We define the \emph{fundamental
  function $\varf_X\colon N\to\br$ of $X$ }by setting
\[
\varf_X(k)=\sup \Big\{ \bignorm{\ts\sum_{i\in A}  x_i}:\, A\subset
N,\ \abs{A}\leq k \Big\}\ ,\qquad k\in N\ .
\]
We then extend the definition of $\varf_X$ to the real interval
$I=\bigcup_{1\leq k<\dim(X)} [k,k+1]$ by linear interpolation. The
fundamental function plays an important r\^ole in the study of so
called greedy bases. Here we shall only need the following facts (see
\eg~\cite{dkosz}*{Section 2}).

\begin{prop}
  \label{prop:fundamental function}
  The functions $t\mapsto\varf_X(t)$ and
  $t\mapsto\frac{t}{\varf_X(t)}$, $t\in I$, are increasing. The
  concave envelope $\psi\colon I\to \br$ of $\varf_X$, \ie
  the (pointwise) smallest concave function dominating $\varf_X$,
  satisfies $\psi(t)\leq 2\varf_X(t)$ for all $t\in I$.
\end{prop}
We next introduce the \emph{lower fundamental function
  $\lambda_X\colon N\to \br$ of $X$ }defined by
\[
\lambda_X(k)=\inf \Big\{ \bignorm{\ts\sum_{i\in A}  x_i}:\, A\subset
N,\ \abs{A}\geq k \Big\}\ ,\qquad k\in N\ ,
\]
and extend the definition to $I$ by linear interpolation. It is clear
that $\lambda_X$ is an increasing function on $I$.

\begin{ex}
Proposition~\ref{prop:properties-of-Fn}(iii)
shows that
\[
\frac{1}{K_p} \cdot \Big(k^{\frac{1}{p}} \meet \tfrac{1}{w}
k^{\frac12} \Big)
\leq \lambda_F(k)\leq \varf_F (k) \leq
k^{\frac{1}{p}} \meet \tfrac{1}{w} k^{\frac12}\ ,
\]
where $F=F^{(n)}_{p,w}$ and $1\leq k\leq n$.
\end{ex}

We now turn to the estimate on the lower fundamental function of
$Y_{p,\vv}$, as promised.

\begin{lem}
  \label{lem:lower-fund-fn-of-Yv}
  Given $1<p<2$, let $\vv=(v_n)$ be a decreasing sequence in $(0,1]$
  such that $v_n\geq n^{-\eta}$ for all $n\in\bn$, where
  $\eta=\frac1{p}-\frac12$. Then for each $k\in\bn$ we have
  \[
  \lambda_{Y_{p,\vv}} (k)\geq \frac1{K_p\cdot\sqrt{2}} \cdot
  \frac1{v_l} \cdot
  l\ ,\qquad \text{where }l=\Bigintp{\sqrt{\tfrac{k}2}}\ .
  \]
  (We put $v_0=1$ to cover the case $k=1$.)
\end{lem}
\begin{proof}
  Let $A$ be a subset of $\{ (n,j):\, n\in\bn,\ 1\leq j\leq n\}$ with
  $\abs{A}\geq k$. For each $n\in\bn$ set $A_n=A\cap \{ (n,j):\, 1\leq
  j\leq n\}$. By Proposition~\ref{prop:properties-of-Fn}(iii) we can
  write $\bn$ as the union of disjoint sets $L$ and $R$, where
  $L=\big\{n\in \bn:\, \lambda_{F_n}(\abs{A_n})\geq
  \frac1{K_p}\abs{A_n}^{\frac1{p}} \big\}$ and $R=\big\{ n\in\bn:\,
  \lambda_{F_n}(\abs{A_n})\geq\frac1{K_p}\frac1{v_n} \abs{A_n}^{\frac12}
  \big\}\setminus L$. Then
  \[
  \biggnorm{\sum_{(n,j)\in A} f^{(n)}_j}_{Y_{p,\vv}} \geq\bigg( \sum_n
  \lambda_{F_n}(\abs{A_n}) ^p \bigg)^{\frac1{p}} \geq
  \frac1{K_p}\cdot\bigg( \sum_{n\in L}
  \abs{A_n} + \sum_{n\in R} \Big(\tfrac1{v_n}
  \abs{A_n}^{\frac12}\Big)^p \bigg)^{\frac1{p}}\ ,
  \]
  and hence, using $p<2$, we obtain
  \begin{equation}
    \label{eq:lower-fund-fn-of-Y}
    K_p\biggnorm{\sum_{(n,j)\in A} f^{(n)}_j}_{Y_{p,\vv}} \geq \bigg(
    \sum_{n\in L} \abs{A_n} + \Big(\sum_{n\in R} \tfrac1{v_n^2}
    \abs{A_n}\Big)^{\frac{p}2} \bigg)^{\frac1{p}}\ .
  \end{equation}
  Set $l=\Bigintp{\sqrt{\frac{k}2}}$. Since
  $\sum_n\abs{A_n}=\abs{A}\geq k$, either $\sum_{n\in L} \abs{A_n}\geq
  l^2$ or $\sum_{n\in R} \abs{A_n}\geq l^2$. In the former case,
  inequality~\eqref{eq:lower-fund-fn-of-Y} immediately gives
  \begin{equation}
    \label{eq:lower-bound-one}
    K_p\biggnorm{\sum_{(n,j)\in A} f^{(n)}_j}_{Y_{p,\vv}} \geq
    l^{\frac2{p}} \geq \frac1{v_{l}}\cdot l\ ,
  \end{equation}
  where the second inequality follows from the assumption that
  $v_n\geq n^{-\eta}$ for all $n\in\bn$. 

  We now consider the case when $\sum_{n\in R} \abs{A_n}\geq
  l^2$. Choose $s\in\bn$ and $0\leq s'\leq s$ such that
  \[
  \sum_{n\in R}\abs{A_n} =\sum_{n=1}^{s-1}n + s'\ .
  \]
  Since $(v_n)$ is decreasing, we have
  \[
  \sum_{n\in R} \tfrac1{v_n^2} \abs{A_n}\geq \sum_{n=1}^{s-1}
  \tfrac1{v_n^2} n + \tfrac1{v_s^2}s' \geq \frac1{v_l^2} \left(
  \sum_{n=l}^{s-1} n + s' \right) \geq \frac1{v_l^2}
  \cdot \frac{l^2}{2}\ ,
  \] 
  where we used $\ds\sum_{n=1}^{l-1} n\leq \frac{l^2}{2}$. Hence
  by~\eqref{eq:lower-fund-fn-of-Y} we obtain
  \begin{equation}
    \label{eq:lower-bound-two}
    K_p\biggnorm{\sum_{(n,j)\in A} f^{(n)}_j}_{Y_{p,\vv}} \geq
    \frac1{\sqrt{2}}\cdot \frac1{v_l}\cdot l\ .
  \end{equation}
  The claim now follows from~\eqref{eq:lower-bound-one}
  and~\eqref{eq:lower-bound-two} above.
\end{proof}

\section{The key lemma}
\label{sec:key-lem}

This section is entirely devoted to a result that will play a central
r\^ole in distinguishing closed ideals. It roughly says that if one
has a bounded operator and the fundamental function of the domain is
asymptotically smaller than the lower fundamental function of the
range space, then a large proportion of basis vectors must map to
`flat' vectors.

\begin{lem}
  \label{lem:key}
  Let $Y$ be an infinite-dimensional Banach space with a normalized,
  1-unconditional basis $(f_j)$. For each $m\in\bn$ let $G_m$ be an
  $m$-dimensional Banach space with a normalized, 1-unconditional
  basis $\big\{ g^{(m)}_i:\,1\leq i\leq m\big\}$. Assume that
  \begin{align}
    \label{eq:varf-sublinear}
    \lim_{k\to\infty}\, \sup_{m\geq k}
    \frac{\varf_{G_m}(k)}{k}&=0\ ,\quad \text{and} \\[2ex]
    \label{eq:lower-vs-upper}
    \lim_{m\to \infty} \frac{\varf_{G_m}(m)}{\lambda_Y(cm)} &= 0\qquad
    \text{for all }c>0\ .
  \end{align}
  If $B_m\colon G_m\to Y$ is a sequence of operators with
  $\sup_m\norm{B_m}\leq 1$, then
  \[
  \frac1{m}\sum_{i=1}^m \bignorm{B_m\big(g^{(m)}_i\big)}_{\infty} \to
  0 \quad\text{as} \quad m\to\infty\ .
  \]
  Here $\norm{y}_\infty=\sup_j \abs{y_j}$ for $y=\sum_j y_j f_j\in
  Y$.
\end{lem}
\begin{rem}
  Before proving our lemma let us look at the extreme case. Assume
  that for each $m\in\bn$ we are given a linear operator $B_m$ from
  $\ell_\infty^m$ to $\ell_1$ with $\norm{B_m}\leq 1$. In that special
  case we can easily deduce our claim from Grothendieck's
  inequality. Indeed, fixing $m\in\bn$, we can write $B_m$ as a matrix
  $B_m=\big(B_m(j,i)\big)$, $i=1,\dots,m,\ j\in\bn$, with
  \[
  \sup \bigg\{\sum_{i=1}^m \sum_{j=1}^\infty t_j B_m(j,i) s_i :
  \abs{s_i}, \abs{t_j}\leq 1\text{ for }1\leq i\leq m,\ 1\leq j<\infty
  \bigg\}=\norm{B_m}\leq 1\ ,
  \]
  and we then have to show that  
  \[
  \frac1m \sum_{i=1}^m \bignorm{B_m\big(e_{\infty,i}^{(m)}\big)}_\infty
  =\frac1m \sum_{i=1}^m \max_{j\in\bn} \bigabs{B_m(j,i)}\to 0\qquad
  \text{as }m\to\infty\ .
  \]
  Now Grothendieck's inequality implies that 
  \[
  \sum_{i=1}^m \sum_{j=1}^\infty  B_m(j,i) \ip{y_j}{x_i}\leq K_G\ ,
  \]
  whenever $x_i$, $i=1,\dots,m$, and $y_j$, $j\in\bn$, are elements of
  the unit ball of a Hilbert space $H$, and where $K_G$ denotes the
  Grothendieck constant. We choose for each $i=1,\dots m$ an integer
  $j_i\in\bn$ such that $\abs{B_m(j_i,i)}=\max_{j\in\bn}
  \abs{B_m(j,i)}$. We then let $H=\ell_2^m$ and $x_i= e^{(m)}_{2,i}$
  for $i=1,\dots m$. For each $j\in\bn$ we define a vector
  $\yt_j=\sum_{i=1}^m \yt_j(i) e^{(m)}_{2,i}$ in $\ell^m_2$ as
  follows.
  \[
  \yt_j(i)=\begin{cases}%
  \sgn\big(B_m(j_i,i)\big) & \text{ if $j=j_i$,}\\ 
  0 & \text{ otherwise.}
  \end{cases}
  \]
  Note that  $\norm{\yt_j}_{\ell_2^m}\leq \sqrt{m}$, and so
  $y_j=\yt_j/\sqrt{m}\in B_{\ell_2^m}$ for each $j\in\bn$. It follows
  that
  \begin{align*}
    K_G &\geq \sum_{i=1}^m\sum_{j=1}^\infty  B_m(j,i) \ip{y_j}{x_i}=
    \sum_{i=1}^m \sum_{j=1}^\infty B_m(j,i) y_j(i)\\[2ex]
    &=\frac1{\sqrt{m}}
    \sum_{i=1}^m \bigabs{B_m(j_i,i)}= \frac1{\sqrt{m}} \sum_{i=1}^m
    \max_{j\in\bn} \bigabs{B_m(j,i)}\ ,
  \end{align*}
  which yields our claim in this special case.
\end{rem}
\begin{proof}[Proof of Lemma~\ref{lem:key}]
  Fix $\varrho>0$. By~\eqref{eq:varf-sublinear} there exists
  $k_0\in\bn$ such that
  \[
  \frac{\varf_{G_m}(k)}{k} < \frac{\varrho}{2}\qquad \text{for all
  }m\geq k\geq k_0\ .
  \]
  Set $c=\varrho k_0^{-1}$. By~\eqref{eq:lower-vs-upper} we may choose
  $m_0\in\bn$ such that
  \[
  m_0>k_0\qquad\text{and}\qquad \frac{\varf_{G_m}(m)}{\lambda_Y(cm)} <
  \varrho\quad \text{for all }m>m_0\ .
  \]
  Now fix $m>m_0$ and set
  \[
  A=\Big\{ i\in\{ 1,2,\dots, m\}:\,
  \bignorm{B_m\big(g^{(m)}_i\big)}_\infty \geq\varrho \Big\}\ .
  \]
  We will show that $\abs{A}\leq\varrho m$. It will then follow that
  \[
  \frac1{m}\sum_{i=1}^m \bignorm{B_m\big(g^{(m)}_i\big)}_{\infty} \leq
  \varrho + \frac{\abs{A}}{m} \leq 2\varrho\ ,
  \]
  and since $m>m_0$ and $\varrho>0$ were arbitrary, the proof of the
  lemma will be
  complete. To show that $\abs{A}\leq\varrho m$ we argue by
  contradiction, and assume that $\abs{A}\geq \varrho m$. For each
  $i\in A$ fix $j_i\in\bn$ such that
  \[
  \biggabs{\Big[B_m\big(g^{(m)}_i\big)\Big]_{j_i}}\geq \varrho\ .
  \]
  Here $[y]_j$ denotes, for $y\in Y$, the $j^{\text{th}}$ coordinate
  of $y$ with respect to the basis $(f_j)_{j=1}^\infty$. We
  then set $\At = \{ j_i:\, i\in A \}$, and for $j\in\At$ we let
  $A_j=\{ i\in A:\, j_i=j\}$. We shall now obtain a number of
  inequalities that will eventually lead to a contradiction.

  Fix $j\in\At$ and for each $i=1,\dots,m$ let $\vare_i$ be the
  sign of $\big[B_m\big(g^{(m)}_i\big)\big]_j$. Since $\norm{B_m}\leq
  1$ and $\big(g^{(m)}_i\big)$ is 1-unconditional, we have
  \begin{align*}
    \varf_{G_m}(\abs{A_j}) & \geq \biggnorm{\sum_{i\in A_j} \vare_i
      g^{(m)}_i}_{G_m} \geq \biggnorm{\sum_{i\in A_j} \vare_i
      B_m\big(g^{(m)}_i\big)}_Y\\[2ex]
    &\geq \bigg[\sum_{i\in A_j} \vare_i B_m\big(g^{(m)}_i\big)\bigg]_j
    \geq \abs{A_j}\varrho\ .
  \end{align*}
  Let $\psi$ be the concave envelope of $\varf_{G_m}$. Since $A$ is
  the disjoint union of the sets $A_j$, $j\in\At$, we obtain
  \begin{eqnarray*}
    \abs{A} &=& \sum_{j\in\At} \abs{A_j} \leq \varrho^{-1}
    \sum_{j\in\At} \varf_{G_m}(\abs{A_j}) \leq \varrho^{-1}
    \sum_{j\in\At} \psi(\abs{A_j}) \\[2ex]
    &\leq& \varrho^{-1} \abs{\At}\cdot \psi\Big(
    \tfrac{\abs{A}}{\abs{\At}} \Big) \leq 2\varrho^{-1} \abs{\At}\cdot
    \varf_{G_m}\Big( \tfrac{\abs{A}}{\abs{\At}} \Big)\ ,
  \end{eqnarray*}
  where the first inequality of the second line uses the concavity of
  $\psi$, and the second inequality uses
  Proposition~\ref{prop:fundamental function}. Now if
  $\frac{\abs{A}}{\abs{\At}}>k_0$, then it follows from the above that
  $\frac{\varf_{G_m}(k_0)}{k_0}\geq \frac{\varrho}2$ which contradicts
  the choice of $k_0$. Thus
  \begin{equation}
    \label{eq:key-lem:A-and-At}
    \abs{A}\leq k_0\abs{\At}\ .
  \end{equation}
  We next fix independent Rademacher random variables $(r_i)_{i\in
    A}$, and show the estimate
  \begin{equation}
    \label{eq:lem:ave-of-coords}
    \be \biggabs{\sum_{i\in A} r_i \Big[ B_m\big(g^{(m)}_i\big) \Big]_j}
    \geq \varrho \qquad \text{for all }j\in \At\ .
  \end{equation}
  To see this fix $j\in\At$ and set $y_i=\big[ B_m\big(g^{(m)}_i\big)
    \big]_j$ for $i\in A$. By definition of $\At$ there is an
  $i_0\in A$ such that $j_{i_0}=j$, and hence
  $\abs{y_{i_0}}\geq\varrho$. Thus
  \begin{eqnarray*}
    \be \biggabs{\sum_{i\in A} r_i y_i} &=& \be \biggabs{\sum_{i\in A}
      r_{i_0} r_i y_i} = \be \biggabs{y_{i_0} + \sum _{i\in A, i\neq
        i_0} r_{i_0}r_i y_i} \\[2ex]
    & \geq & \biggabs{y_{i_0}  + \sum _{i\in A, i\neq i_0} \be
      (r_{i_0}r_i) y_i} = \abs{y_{i_0}}\geq \varrho\ ,
  \end{eqnarray*}
  using Jensen's inequality at the start of the second line. We next
  obtain
  \begin{eqnarray*}
    \varf_{G_m} (\abs{A}) &\geq & \be \biggnorm{\sum_{i\in A} r_i
      B_m\big(g^{(m)}_i\big) }_Y \qquad \text{as $\norm{B_m}\leq
      1$,}\\[2ex]
    & = & \be \biggnorm{\sum_j \biggabs{\sum_{i\in A} r_i \Big[
          B_m\big(g^{(m)}_i\big) \Big]_j} f_j}_Y \quad\text{as $(f_j)$
      is 1-unconditional,}\\[2ex]
    &\geq & \biggnorm{\sum_j \be \biggabs{\sum_{i\in A} r_i \Big[
          B_m\big(g^{(m)}_i\big) \Big]_j} f_j}_Y \qquad \text{by
      Jensen's inequality,}\\[2ex]
    &\geq & \varrho \biggnorm{\sum_{j\in\At}f_j}_Y
    \quad\text{using \eqref{eq:lem:ave-of-coords} and
      1-unconditionality of $(f_j)$ ,}\\[2ex]
    &\geq & \varrho \lambda_Y(\abs{\At})\ .
  \end{eqnarray*}
  Recall that $c=\varrho k_0^{-1}$ and $A\subset \{1,\dots, m\}$ with
  $\abs{A}\geq \varrho m$. So
  $\abs{A}\leq m$, and by~\eqref{eq:key-lem:A-and-At}, $\abs{\At}\geq
  cm$. Thus, the above gives
  \[
  \varrho\leq \frac{\varf_{G_m}(\abs{A})}{\lambda_Y(\abs{\At})} \leq
  \frac{\varf_{G_m}(m)}{\lambda_Y(cm)} <\varrho
  \]
  by the choice of $m_0$. This contradiction completes the proof.
\end{proof}

\section{Proof of the main results}
\label{sec:main-results}

In the previous section we described a situation when images of basis
vectors are on average `flat'. Here we begin with a calculation
(Lemma~\ref{lem:F_n-to-ell_2^n} below) that shows that certain formal
inclusion maps reduce the norm of `flat'
vectors. We then introduce, in the special case $1<p<2$ and
$p<q<\infty$,  a class of closed ideals in $\cL(\ell_p,\ell_q)$
parametrised by decreasing sequences in $(0,1]$.
Theorem~\ref{thm:main} and Corollary~\ref{cor:distinct-ideals} show
when these ideals are distinct from one another. This will lead to a
proof of our main result, Theorem~\ref{mainthm:infinite}. In the rest
of the section we follow a similar strategy and establish
Theorem~\ref{mainthm:refinement}.
\begin{lem}
  \label{lem:F_n-to-ell_2^n}
  Given $1<p<2$ and $p<q<\infty$, let $n\in\bn$, $v\in(0,1]$, and
  $F=F^{(n)}_{p,v}$ with basis $\big\{ f^{(n)}_j:\,1\leq
    j\leq n\big\}$. Let $y=\sum_{j=1}^n y_jf^{(n)}_j\in F$ with
  $\norm{y}_F\leq 1$, and let $\yt=\sum_{j=1}^n
  y_je^{(n)}_{2,j}\in\ell_2^n$. If
  $\norm{y}_\infty=\max_j\abs{y_j}\leq\sigma\leq 1$ and $v\leq
  \sigma^{\frac12-\frac1{p'}}$, then
  \[
  \norm{\yt}_{\ell_2^n}^q \leq \max\{C_p^p,K_p^q\}\cdot \sigma^r\cdot
  \norm{y}_F^p\ ,
  \]
  where $C_p$ is the cotype-2 constant of $\ell_p$ and $r=\min\big\{
  \frac{q}2-\frac{p}2\ ,\ \frac{q}2-\frac{q}{p'}\big\}$.
\end{lem}
Here we recall that for $1\leq p\leq 2$ the Banach space $L_p[0,1]$,
and hence $\ell_p$, has cotype~2. This is a consequence of
Khintchine's inequality. See for example~\cite{lt:79}*{Definition
  1.e.12}.
\begin{proof}
  If $\norm{\yt}_{\ell_2^n}\leq \sqrt{\sigma}$, then
  \[
  \label{eq:norm-of-comp-1}
  \norm{\yt}_{\ell_2^n}^q = \norm{\yt}_{\ell_2^n}^{q-p}\cdot
  \norm{\yt}_{\ell_2^n}^p \leq \sigma^{\frac{q}2-\frac{p}2}\cdot
  \norm{\yt}_{\ell_2^n}^p \leq C_p^p \cdot\sigma^{\frac{q}2-\frac{p}2}
  \cdot \norm{y}_{F}^p\ ,
  \]
  where we use the fact, an easy consequence of the definition of
  cotype, that in $\ell_p$ a normalized, 1-unconditional basis
  $C_p$-dominates the unit vector basis of $\ell_2$. So the claim
  holds in this case.

  Now assume that $\norm{\yt}_{\ell_2^n}>\sqrt{\sigma}$. Set
  $z_j=\frac{y_j}{\norm{\yt}_{\ell_2^n}}$ for $1\leq j\leq n$, and
  let $\zt=\sum_{j=1}^n z_j e^{(n)}_{2,j}$. Then $\norm{\zt}_{\ell_2^n}=1$ and
  $\ip{\yt}{\zt}=\norm{\yt}_{\ell_2^n}$. Note also that $\abs{z_j}\leq
  \sqrt{\sigma}$ for all $j$. So we have
  \[
  \left(\sum_{j=1}^n
  \abs{z_j}^{p'}\right)^{\frac1{p'}} = \left(\sum_{j=1}^n
  \abs{z_j}^{p'-2}\cdot \abs{z_j}^2 \right)^{\frac1{p'}}
  \leq \sigma^{\frac12-\frac1{p'}} \cdot \left(\sum_{j=1}^n
  \abs{z_j}^2\right)^{\frac1{p'}}  =
  \sigma^{\frac12-\frac1{p'}}\ .
  \]
  On the other hand, we have $v\cdot\norm{\zt}_{\ell_2^n}=v\leq
  \sigma^{\frac12-\frac1{p'}}$ by assumption. It follows that
  $\norm{z}_{p',v}\leq \sigma^{\frac12-\frac1{p'}}$, where
  $z=\sum_{j=1}^n z_j e^{(n)}_j$, and $\big\{ e^{(n)}_j:\,1\leq j\leq n
  \big\}$ is the unit vector basis of
  $E=E^{(n)}_{p',v}=(\br^n,\norm{\cdot}_{p',v})$. Hence
  by~\eqref{eq:xp-isom-p<2-n} we have
  \[
  \norm{\yt}_{\ell_2^n} = \ip{\yt}{\zt} \leq
  \norm{(y_j)}^*_{p',v}\cdot \norm{(z_j)}_{p',v} \leq
  K_p\cdot\sigma^{\frac12-\frac1{p'}}  \cdot \norm{y}_{F}\ .
  \]
  It follows that $\norm{\yt}_{\ell_2^n}^q\leq K_p^q\cdot
  \sigma^{\frac{q}2-\frac{q}{p'}}  \cdot \norm{y}_{F}^q$, which proves
  the claim since $\norm{y}_F\leq 1$, and so $\norm{y}_F^q\leq
  \norm{y}_F^p$.
\end{proof}

Fix $1<p<2$ and $p<q<\infty$. Let $\vv=(v_n)$ be a decreasing sequence
in $(0,1]$. In
Section~\ref{subsec:the-space-Yv} we introduced the complemented
subspace $Y_{\vv}=Y_{p,\vv}$ of $\ell_p$ with corresponding projection
$P_{\vv}=P_{p,\vv}$. As already mentioned in the proof above, for each
$n\in\bn$, the unit vector basis of $\ell_2^n$ is 
$C_p$-dominated by the normalized, 1-unconditional basis
$\big\{ f^{(n)}_j:\, 1\leq j\leq n\big\}$ of
$F_n=F^{(n)}_{p,v_n}$. Thus the formal inclusion map
\[
I_{Y_{\vv},Z_q}\colon Y_{\vv}= \left(\bigoplus _{n=1}^\infty F_n
\right)_{\ell_p}\to Z_q=\left( \bigoplus_{n=1}^\infty \ell_2^n
\right)_{\ell_q}
\]
given by $I_{Y_{\vv},Z_q}\big(f^{(n)}_j\big) = e^{(n)}_{2,j}$ is
well-defined and bounded. This defines the closed ideal
$\cJ^{I_{Y_{\vv},Z_q}}$ of $\cL(\ell_p,\ell_q)$ generated by operators
factoring through $I_{Y_{\vv},Z_q}$.  We also fixed, in
Section~\ref{subsec:lp-spaces}, an
isomorphism $U_q\colon Z_q\to\ell_q$. Note that the operator
$T_{\vv}=U_q\circ I_{Y_{\vv},Z_q}\circ P_{\vv}$ belongs to the ideal
$\cJ^{I_{Y_{\vv},Z_q}}$. The next result establishes conditions on two
sequences $\vv$ and $\vw$ which imply that $T_{\vw}\notin
\cJ^{I_{Y_{\vv},Z_q}}$.
\begin{thm}
  \label{thm:main}
  Fix $1<p<2$ and $p<q<\infty$.
  Let $\vv=(v_n)$ and $\vw=(w_n)$ be decreasing sequences in
  $(0,1]$. Consider $Y_{\vv}$, $I_{Y_{\vv},Z_q}$ and $T_{\vw}$ as
  above. Assume that $v_n\geq n^{-\eta}$ and $w_n\geq n^{-\eta}$ for
  all $n\in\bn$, where $\eta=\frac1{p}-\frac12$. Further assume that
  \begin{equation}
    \label{eq:w-bigger-than-v}
    \frac{v_{\sqrt{cn}}}{w_n}\to 0\quad\text{as}\quad n\to\infty\quad
    \text{for all }c\in (0,1)
  \end{equation}
  (where we simplify notation by letting $v_x=v_{\intp{x}}$ for
  $x\in\br$). Then $T_{\vw}\notin \cJ^{I_{Y_{\vv},Z_q}}$.
\end{thm}
\begin{proof}
  For each $n\in\bn$ let $F_n=F^{(n)}_{p,v_n}$ and
  $G_n=F^{(n)}_{p,w_n}$ with unit vector bases $\big\{
  f^{(n)}_j:\, 1\leq j\leq n\big\}$ and $\big\{ g^{(n)}_j:\, 1\leq
  j\leq n\big\}$, respectively. Thus $f^{(n)}_j=f^{(n)}_{p,v_n,j}$ and
  $g^{(n)}_j=f^{(n)}_{p,w_n,j}$ for $1\leq j\leq n$ using the notation
  introduced in Section~\ref{subsec:rosenthal-xp}.
  To simplify notation we write
  $Y=Y_{\vv}$, $Z=Z_q$, $U=U_q$ and $T=T_{\vw}$. Thus $U\colon
  Z\to\ell_q$ is an isomorphism,
  \[
  I_{Y,Z}\colon Y= \left(\bigoplus _{n=1}^\infty F_n
  \right)_{\ell_p}\to Z=\left( \bigoplus_{n=1}^\infty \ell_2^n
  \right)_{\ell_q}
  \]
  is given by $I_{Y,Z}\big(f^{(n)}_j\big) = e^{(n)}_{2,j}$, and
  $T\colon \ell_p=\big(\bigoplus _{m=1}^\infty \ell_p^{k_m}\big)_{\ell_p}
  \to \ell_q$ is the composite $T=U\circ I_{Y_{\vw},Z}\circ
  P_{\vw}$. Note that $T\big(g^{(m)}_i\big)=U \big( e^{(m)}_{2,i}\big)$.

  We need to show that $T\notin \cJ^{I_{Y,Z}}$. We achieve this by
  finding a separating functional $\Phi\in\cL(\ell_p,\ell_q)^*$ as
  follows. For each $m\in\bn$ we define $\Phi_m\in
  \cL(\ell_p,\ell_q)^*$ by setting
  \[
  \Phi_m(V)= \frac1{m} \sum_{i=1}^m
  \bigip{e^{(m)}_{2,i}}{U^{-1}V(g^{(m)}_i)}\ ,\qquad
  V\in\cL(\ell_p,\ell_q)\ .
  \]
  Since $\norm{\Phi_m}\leq \norm{U^{-1}}$ for all $m$, the sequence
  $(\Phi_m)$ has a $\omega^*$-accumulation point $\Phi$ in
  $\cL(\ell_p,\ell_q)^*$. Note that $\Phi_m(T)=1$ for all $m$, and
  hence $\Phi(T)=1$. The proof will be complete if we can show that
  $\cJ^{I_{Y,Z}}$ is contained in the kernel of $\Phi$. To see this,
  fix $A\in \cL(Z,\ell_q)$ and $B\in\cL(\ell_p,Y)$ with $\norm{A}\leq
  1$ and $\norm{B}\leq 1$. It is sufficient to show that
  \begin{equation}
    \label{eq:thm-main:aim}
    \Phi_m(AI_{Y,Z}B)\to 0\qquad\text{as}\qquad m\to\infty\ .
  \end{equation}
  Let $B_m\colon G_m\to Y$ denote the restriction to $G_m$ of $B$. We shall
  use Lemma~\ref{lem:key} to show that
  \begin{equation}
    \label{eq:thm-main:aim2}
    \frac1{m}\sum_{i=1}^m \bignorm{B_m\big(g^{(m)}_i\big)}_{\infty}
    \to 0 \quad\text{as} \quad m\to\infty\ .
  \end{equation}
  Recall that for $y=\sum_{n=1}^\infty \sum_{j=1}^n y_{n,j}f^{(n)}_j$
  in $Y$ we let $\norm{y}_\infty=\sup_{n\in\bn,\ 1\leq j\leq n} \abs{y_{n,j}}$.
  By Proposition~\ref{prop:properties-of-Fn}(iii) we have
  $\varf_{G_m}(k)\leq k^{\frac1{p}}$ for all $1\leq k \leq m$, and so
  condition~\eqref{eq:varf-sublinear} of Lemma~\ref{lem:key}
  certainly holds. Now by Lemma~\ref{lem:lower-fund-fn-of-Yv} we
  have
  \[
  \lambda_{Y} (k)\geq \frac1{3K_p} \cdot \frac1{v_{\sqrt{k/2}}} \cdot
  \sqrt{k}
  \]
  for all large $k$. On the other hand, by
  Proposition~\ref{prop:properties-of-Fn}(iii) we
  have $\varf_{G_m}(m)\leq\frac{1}{w_m} m^{\frac12}$ for all $m$. So for
  any $c>0$ and for all large $m\in\bn$, it follows that
  \[
  \frac{\varf_{G_m}(m)}{\lambda_Y(cm)} \leq C\cdot
  \frac{v_{\sqrt{c'm}}}{w_m}\ ,
  \]
  where $C$ and $c'$ are constants depending
  only on $c$ (and $p$). Thus, by assumption~\eqref{eq:w-bigger-than-v},
  condition~\eqref{eq:lower-vs-upper} also holds,
  and Lemma~\ref{lem:key} applies. This completes the proof
  of~\eqref{eq:thm-main:aim2}. To see~\eqref{eq:thm-main:aim}, fix
  $\varrho\in(0,1)$ and choose $n_0\in\bn$ such that $v_n\leq
  \varrho^{\frac12-\frac1{p'}}$ for all $n\geq n_0$. This is possible,
  since by~\eqref{eq:w-bigger-than-v} we have $v_n\to 0$ as
  $n\to\infty$. Note that
  \begin{equation}
    \label{eq:bound-on-Phi-m}
    \begin{aligned}
      \abs{\Phi_m(AI_{Y,Z}B)} &= \frac1{m} \sabs{\sum_{i=1}^m
        \Bigip{A^*\big(U^{-1}\big)^*e^{(m)}_{2,i}}{I_{Y,Z}B_m\big(g^{(m)}_i\big)}}
      \\[2ex]
      &\leq \frac1{m} \sum_{i=1}^m \bignorm{U^{-1}}\cdot
      \bignorm{I_{Y,Z}B_m\big(g^{(m)}_i\big)}_{Z}\ .
    \end{aligned}
  \end{equation}
  Fix $1\leq i\leq m$, and set $x=B_m\big(g^{(m)}_i\big)$. Write
  $x=\sum_{n=1}^\infty \sum_{j=1}^n x_{n,j}f^{(n)}_j$, and let
  $\sigma=\varrho\join\norm{x}_\infty$. Note that $v_n\leq
  \sigma^{\frac12-\frac1{p'}}$ for all
  $n\geq n_0$. Hence by Lemma~\ref{lem:F_n-to-ell_2^n} we have
  \[
  \left(\sum_{j=1}^n \abs{x_{n,j}}^2\right)^{\frac{q}2} \leq
  M\cdot \sigma^r\cdot \snorm{\sum_{j=1}^n x_{n,j}
    f^{(n)}_j}_{F_n}^p\qquad \text{for all }n\geq n_0\ ,
  \]
  where $M=\max\{C_p^p,K_p^q\}$ and $r=\min\big\{
  \frac{q}2-\frac{p}2\ ,\ \frac{q}2-\frac{q}{p'}\big\}$. It follows
  that
  \begin{align*}
    \bignorm{I_{Y,Z}B_m\big( &g^{(m)}_i\big)}_{Z} =   
    \norm{I_{Y,Z}(x)} = \left(\sum _{n=1}^\infty \bigg(\sum_{j=1}^n
    \abs{x_{n,j}}^2\bigg)^{\frac{q}2}\right)^{\frac1{q}}\\[2ex]
    &\leq \left(\sum _{n=1}^{n_0} \bigg(\sum_{j=1}^n
    \abs{x_{n,j}}^2\bigg)^{\frac{q}2}\right)^{\frac1{q}} +
    M^{\frac1{q}}\cdot \sigma^{\frac{r}{q}} \cdot \left( 
    \sum_{n>n_0} \biggnorm{\sum_{j=1}^n x_{n,j} f^{(n)}_j}_{F_n}^p
    \right)^{\frac1{q}} \\[2ex]
    &\leq \norm{x}_\infty\cdot \left( \sum_{n=1}^{n_0} n^{\frac{q}2}
    \right)^{\frac1{q}} + M^{\frac1{q}}\cdot \sigma^{\frac{r}{q}}
    \cdot \norm{x}_Y^{\frac{p}{q}} \leq \norm{x}_\infty\cdot N +
    M^{\frac1{q}}\cdot \sigma^{\frac{r}{q}}\\[2ex]
    & =  N\cdot \bignorm{B_m\big( g^{(m)}_i\big)}_\infty +
    M^{\frac1{q}}\cdot \sigma^{\frac{r}{q}}\ ,
  \end{align*}
  where we put $N=\big( \sum_{n=1}^{n_0} n^{\frac{q}2}
  \big)^{\frac1{q}}$. Hence, using~\eqref{eq:bound-on-Phi-m} and
  putting $\sigma^{(m)}_i=\varrho\join
  \bignorm{B_m\big(g^{(m)}_i\big)}_\infty$ for all $m\in\bn$ and
  $1\leq i\leq m$, we obtain
  \begin{align*}
    \bignorm{U^{-1}}^{-1} \cdot \abs{\Phi_m(AI_{Y,Z}B)} &\leq
    N\cdot\frac1{m}\sum _{i=1}^m
    \bignorm{B_m\big(g^{(m)}_i\big)}_\infty +  M^{\frac1{q}}\cdot \frac1{m}
    \sum_{i=1}^m \big(\sigma^{(m)}_i\big)^{\frac{r}{q}} \\[2ex]
    &\leq N\cdot\frac1{m}\sum _{i=1}^m
    \bignorm{B_m\big(g^{(m)}_i\big)}_\infty
    +  M^{\frac1{q}}\cdot \left(\frac1{m} \sum_{i=1}^m
    \sigma^{(m)}_i\right)^{\frac{r}{q}}\ .
  \end{align*}
  To see the second inequality note that $\frac{r}{q}<\frac12$, and so
  the function $t\mapsto t^{\frac{r}{q}}$ is concave. Now it follows
  from~\eqref{eq:thm-main:aim2} that
  \[
  \limsup_m \abs{\Phi_m(AI_{Y,q}B)} \leq \bignorm{U^{-1}} \cdot
  M^{\frac1{q}}\cdot \varrho^{\frac{r}{q}}\ .
  \]
  Since $\varrho>0$ was arbitrary, the proof
  of~\eqref{eq:thm-main:aim}, and hence of the theorem is complete.
\end{proof}

\begin{cor}
  \label{cor:distinct-ideals}
  Let $1<p<2$ and $p<q<\infty$. Let $\vv=(v_n)$ and $\vw=(w_n)$ be
  decreasing sequences in $(0,1]$ bounded below by $n^{-\eta}$,
  $\eta=\frac1{p}-\frac12$, and satisfying
  condition~\eqref{eq:w-bigger-than-v}. Then
  $\cJ^{I_{Y_{\vv},Z_q}}\subsetneq \cJ^{I_{Y_{\vw},Z_q}}$.
\end{cor}
\begin{proof}
  It follows from~\eqref{eq:w-bigger-than-v} that eventually $v_n\leq
  w_n$. Hence, using the notation of the proof of
  Theorem~\ref{thm:main}, the basis $\big(f^{(n)}_j\big)$ of $F_n$
  $K_p$-dominates the basis $\big(g^{(n)}_j\big)$ of $G_n$ for all
  large $n$. Indeed, this follows from~\eqref{eq:xp-isom-p<2-n}. Thus,
  $I_{Y_{\vv},Z_q}$ factors through $I_{Y_{\vw},Z_q}$ via the formal
  inclusion map $I_{Y_{\vv},Y_{\vw}}$, and thus
  $\cJ^{I_{Y_{\vv},Z_q}}\subset \cJ^{I_{Y_{\vw},Z_q}}$. The claim now
  follows immediately from Theorem~\ref{thm:main} since $T_{\vw}\in
  \cJ^{I_{Y_{\vw},Z_q}}$.
\end{proof}

Before proving Theorem~\ref{mainthm:infinite} we need to show that
certain maps are finitely strictly singular.

\begin{prop}
  \label{prop:inclusion-fss}
  Fix $1<p<2$ and $p<q<\infty$. Let $\vv=(v_n)$ be a decreasing
  sequence in $(0,1]$ bounded below by $n^{-\eta}$,
  $\eta=\frac1{p}-\frac12$, such that $v_n\to 0$ as $n\to\infty$. Set
  $Y=Y_{p,\vv}$. Then the formal inclusion maps $I_{Y,Z_q}$ and
  $I_{Z_{q'},Y^*}$ are finitely strictly singular.
\end{prop}
\begin{proof}
  For each $n\in\bn$ let $F_n=F^{(n)}_{p,v_n}$ and
  $E_n=\big(\br^n,\norm{\cdot}_{p',v_n}\big)$ with unit vector
  bases $\big\{ f^{(n)}_j:\, 1\leq j\leq n\big\}$ and $\big\{
  e^{(n)}_j:\, 1\leq j\leq n\big\}$, respectively. We first prove that
  $I_{Y,Z_q}$ is finitely strictly singular. Fix $\vare>0$. Choose
  $\varrho\in (0,1)$ such
  that $\varrho + M^{\frac1{q}}\cdot \varrho^{\frac{r}{q}}<\vare$,
  where, as before, $M=\max\{ C_p^p,K_p^q\}$,
  $r=\min\big\{\frac{q}2-\frac{p}2\ ,\ \frac{q}2-\frac{q}{p'}\big\}$
  and $C_p$ is the type-2 constant of $\ell_p$. Next fix $n_0\in\bn$
  such that $v_n\leq \varrho^{\frac12-\frac1{p'}}$ for all $n\geq
  n_0$. Set $N=\big( \sum_{n=1}^{n_0} n^{\frac{q}2}
  \big)^{\frac1{q}}$. Finally, choose $d\in\bn$ such that
  $K_p\cdot\frac{2v_d}{d}\cdot N < \varrho$.

  Now let $H$ be a subspace of $Y$ of dimension at least $2d^2$. Then
  by a result of V.~D.~Milman~\cite{milman:70} (see
  also~\cite{sari-schlump-tomczak-troitsky:07}*{Lemma 3.4}), there
  exists $x=\sum_{n=1}^\infty \sum_{j=1}^n x_{n,j} f^{(n)}_j\in H$
  such that $\abs{x_{m,i}}=\norm{x}_\infty=\sup_{n\in\bn,\ 1\leq j\leq
    n} \abs{x_{n,j}}$ for at least $2d^2$ pairs $(m,i)$. Hence by
  Lemma~\ref{lem:lower-fund-fn-of-Yv}, assuming as we may that
  $\norm{x}_Y=1$, we have
  \begin{equation}
    \label{eq:sup-norm-estimate}
    1=\norm{x}_Y\geq \norm{x}_\infty\cdot \lambda_Y\big(2d^2\big) \geq
    \norm{x}_\infty\cdot \frac1{K_p\sqrt{2}}\cdot \frac{d}{v_d}\ .
  \end{equation}
  Set $\sigma=\varrho\join \norm{x}_\infty$. As in the proof of
  Theorem~\ref{thm:main}, we obtain
  \[
  \norm{I_{Y,Z}(x)} \leq N\cdot \norm{x}_\infty +
  M^{\frac1{q}}\cdot \sigma^{\frac{r}{q}}\ .
  \]
  By~\eqref{eq:sup-norm-estimate} and by the choice of $d$, we have
  $N\cdot \norm{x}_{\infty}\leq K_p\cdot\frac{2v_d}{d}\cdot N < \varrho$. In
  particular, $\sigma=\varrho$, and hence the above gives
  \[
  \norm{I_{Y,Z}(x)} \leq \varrho +  M^{\frac1{q}}\cdot
  \varrho^{\frac{r}{q}} <\vare
  \]
  by the choice of $\varrho$.

  \noindent
  The proof for $I_{Z_{q'},Y^*}$ is similar. One first needs to prove
  a dual version
  of Lemma~\ref{lem:F_n-to-ell_2^n}, which is easier since in $E_n$
  we have an explicit formula for the norm, and then one needs to
  obtain a series of estimates as in the proof of
  Theorem~\ref{thm:main}. We first observe that $Y^*$ is isomorphic to
  $W=\big( \bigoplus_{n=1}^\infty E_n \big)_{\ell_{p'}}$
  by~\eqref{eq:xp-isom-p<2-n}, and so it is
  sufficient to show that the formal inclusion map $I_{Z_{q'},W}$ is
  finitely strictly singular. So let us fix $\vare>0$, and then choose
  $\varrho>0$ such that $\varrho+\varrho^{1-\frac{q'}{p'}}<\vare$. We
  may and shall assume that $p<q\leq 2$. Indeed, given $p<q_1<q_2$, we
  have $I_{Z_{q_2'},W}=I_{Z_{q_1'},W}\circ I_{Z_{q_2'},Z_{q_1'}}$. Now
  choose $n_0\in\bn$ such that $v_n<\varrho^{1-\frac{q'}{p'}}$ for
  all $n\geq n_0$. Set $N=\big( \sum_{n=1}^{n_0} n^{p'}
  \big)^{\frac1{p'}}$ and choose $d\in\bn$ with $d^{-\frac1{q'}}\cdot
  N<\varrho$.

  Given a subspace $H$ of $Z_{q'}$ of dimension at least $d$, use
  Milman's lemma again to find $x=\sum_{n=1}^\infty \sum_{j=1}^n
  x_{n,i} e^{(n)}_{2,j}\in H$ with $\norm{x}_{Z_{q'}}=1$ such that
  $\abs{x_{m,i}}=\norm{x}_\infty=\sup_{n\in\bn,\ 1\leq j\leq n}
  \abs{x_{n,j}}$ for at least $d$ pairs $(m,i)$. Since $2\leq q'$, we
  have
  \begin{equation}
    \label{eq:sup-norm-estimate2}  
    1=\norm{x}_{Z_{q'}} \geq \norm{x}_{\ell_{q'}}\geq \norm{x}_\infty
    \cdot d^{\frac1{q'}}\ \text{, and so }\norm{x}_\infty\leq
    d^{-\frac1{q'}} < \varrho\ .
  \end{equation}
  Now fix $n\in\bn$ with $n\geq n_0$. On the one hand, we have
  \[
  \sum_{j=1}^n \abs{x_{n,j}}^{p'} = \sum_{j=1}^n \abs{x_{n,j}}^{p'-q'}
  \abs{x_{n,j}}^{q'} \leq \varrho ^{p'-q'} \cdot\left( \sum_{j=1}^n
  \abs{x_{n,j}}^2 \right)^{\frac{q'}2}\ ,
  \]
  where we used $\norm{x}_\infty <\varrho$ and that $2\leq q'$. On the
  other hand, by the choice of $n_0$, and since $q'<p'$, we have
  \[
  v_n^{p'} \left( \sum_{j=1}^n \abs{x_{n,j}}^2 \right)^{\frac{p'}2}
  \leq \varrho^{p'-q'}\cdot \left( \sum_{j=1}^n \abs{x_{n,j}}^2
  \right)^{\frac{q'}2}\ .
  \]
  The previous two inequalities imply that
  \[
  \biggnorm{\sum_{j=1}^n x_{n,j} e^{(n)}_j}_{E_n}^{p'}\leq  \varrho ^{p'-q'}\cdot
  \biggnorm{\sum_{j=1}^n x_{n,j} e^{(n)}_{2,j}}_{\ell_2^n}^{q'}\ .
  \]
  We deduce the following estimates.
  \begin{align*}
    \norm{I_{Z_{q'},W}(x)} &= \left( \sum_{n=1}^\infty
    \biggnorm{\sum_{j=1}^n x_{n,j} e^{(n)}_j}_{E_n}^{p'}
    \right)^{\frac1{p'}} \\[2ex]
    &\leq \left( \sum_{n=1}^{n_0} \biggnorm{\sum_{j=1}^n x_{n,j}
      e^{(n)}_j}_{E_n}^{p'} \right)^{\frac1{p'}} +
    \varrho^{1-\frac{q'}{p'}} \cdot \left( \sum_{n>n_0}
    \biggnorm{\sum_{j=1}^n x_{n,j} e^{(n)}_{2,j}}_{\ell_2^n}^{q'}
    \right)^{\frac1{p'}}\\[2ex]
    &\leq \norm{x}_{\infty} \cdot N + \varrho^{1-\frac{q'}{p'}} \cdot
    \norm{x}_{Z_{q'}}^{\frac{q'}{p'}} \leq \varrho + \varrho^{1-\frac{q'}{p'}}
    <\vare\ ,
  \end{align*}
  where we recall $N=\big( \sum_{n=1}^{n_0} n^{p'}
  \big)^{\frac1{p'}}$, and we used~\eqref{eq:sup-norm-estimate2}, the
  choice of $d$, and the choice of $\varrho$.
\end{proof}

We are now ready to prove our main result.

\begin{proof}[Proof of Theorem~\ref{mainthm:infinite}]
  We first consider the case $1<p<2$ and $p<q<\infty$.
  Put $\eta=\frac1{p}-\frac12$, and define $f\colon \bn\to\br$ by
  setting $f(n)=n^{-\eta}$ for each $n\in\bn$. For an infinite set
  $M\subset \bn$ we define a decreasing sequence
  $\vw_M=\big(\vw_M(n)\big)_{n=1}^\infty$ in $(0,1]$ as follows. Let
  $m_1<m_2<\dots$ be the elements of $M$. We set $\vw_M(1)=1$, $\vw_M\big(
  2^{3^{m_k}}\big) = f(2^k)$ for each $k\in\bn$, and then extend the
  definition of $\vw_M$ to the rest of $\bn$ by linear interpolation. It
  is clear that $\vw_M(n)\geq n^{-\eta}$ for all $n\in\bn$. We
  will show that for infinite sets $M\subset N\subset\bn$ with
  $N\setminus M$ also infinite, the sequences $\vv=\vw_N$ and $\vw=\vw_M$
  satisfy condition~\eqref{eq:w-bigger-than-v} of
  Theorem~\ref{thm:main}. Hence, by
  Corollary~\ref{cor:distinct-ideals} we will have
  $\cJ^{I_{Y_{\vv},Z_q}}\subsetneq \cJ^{I_{Y_{\vw},Z_q}}$. Let us
  first explain how we complete the proof of our main theorem from
  here. We fix a chain $\cC$ of size the continuum consisting of
  infinite subsets of $\bn$ with any two having infinite
  difference. For $M\in\cC$ put $Y_M=Y_{p,\vw_M}$. By above, the closed
  ideals $\cJ^{I_{Y_M,Z_q}}$, $M\in\cC$, are pairwise distinct and
  comparable. Moreover, for each $M\in\cC$, the operator $I_{Y_M,Z_q}$
  is finitely strictly singular by
  Proposition~\ref{prop:inclusion-fss}, and is clearly not
  compact. Hence the ideal $\cJ^{I_{Y_M,Z_q}}$ lies between
  $\cJ^{I_{p,q}}$ and $\fss$.

  Using the same $p$ and $q$, we now consider ideals in
  $\cL(\ell_{q'},\ell_{p'})$. For each $M\in\cC$, we have
  $I_{Z_{q'},Y_M^*}=I_{Y_M,Z_{q}}^*$, and hence by simple duality we
  have
  \[
  \cJ^{I_{Z_{q'},Y_M^*}}(\ell_{q'},\ell_{p'}) = \big\{ T^*:\,
  T\in\cJ^{I_{Y_M,Z_{q}}} \big\}\ .
  \]
  Thus $\big\{\cJ^{I_{Z_{q'},Y_M^*}}:\,M\in\cC\big\}$ is a chain of
  closed ideals in $\cL(\ell_{q'},\ell_{p'})$ of size the
  continuum. By Proposition~\ref{prop:inclusion-fss}, the operators
  $I_{Z_{q'},Y_M^*}$, $M\in\cC$, are finitely strictly singular, and
  clearly not compact, so these ideals also lie between
  $\cJ^{I_{q',p'}}$ and $\fss$. Since $q'<p'$ and $2<p'$, we have
  covered all remaining cases.

  Let us now return to our claim: for infinite sets $M\subset
  N\subset\bn$ with $N\setminus M$ also infinite, the sequences
  $\vv=\vw_N$ and $\vw=\vw_M$ satisfy condition~\eqref{eq:w-bigger-than-v}
  of Theorem~\ref{thm:main}. Fix
  $l\in\bn$. We will show that for all sufficiently large $n\in\bn$,
  we have $\frac{\vw_N\big(n^{\frac13}\big)}{\vw_M(n)}\leq 2^{-\eta l}$,
  which proves the claim. Since $N\setminus M$ is infinite, there
  exists $k_0\in\bn$ such that for all $n\geq m_{k_0}$ we have
  $\abs{N\cap\{1,\dots,n\}}\geq \abs{M\cap\{1,\dots,n\}}+l+2$. Fix
  $n>2^{3^{m_{k_0+1}}}$. This defines $k>k_0$ such that
  $2^{3^{m_k}}\leq n <2^{3^{m_{k+1}}}$. It follows that
  \begin{equation}
    \label{eq:lower-bound-on-w}
    \vw_M(n)\geq \vw_M \big( 2^{3^{m_{k+1}}} \big) = f(2^{k+1})\ .
  \end{equation}
  Next define $k'\in\bn$ by $n_{k'}\leq m_k-1<n_{k'+1}$. By the choice
  of $k_0$, and since $m_k-1\geq m_{k_0}$, we have
  \[
  k'=\abs{N\cap \{1,\dots,m_k-1\}} \geq \abs{M\cap \{1,\dots,m_k-1\}}
  +l+2 = k+l+1\ .
  \]
  It follows that
  \begin{equation}
    \label{eq:upper-bound-on-v}
    \vw_N\big(n^{\frac13}\big) \leq \vw_N\big(2^{3^{m_k-1}}\big)  \leq
    \vw_N \big( 2^{3^{n_{k'}}} \big) =
    f(2^{k'}) \leq f(2^{k+l+1})\ .
  \end{equation}
  Putting together~\eqref{eq:lower-bound-on-w}
  and~\eqref{eq:upper-bound-on-v} yields
  \[
  \frac{\vw_N\big(n^{\frac13}\big)}{\vw_M(n)} \leq
  \frac{f(2^{k+l+1})}{f(2^{k+1})} = 2^{-\eta l}\ .
  \]
  This holds for any $n\geq m_{k_0}$, so the proof of our claim is
  complete.
\end{proof}

We conclude this section with a proof of
Theorem~\ref{mainthm:refinement}. This will be very similar to the
general case but simpler. We shall still rely on our key lemma from
Section~\ref{sec:key-lem}. From now on we fix $1<p<2<q<\infty$.
As usual, for a decreasing sequence $\vv=(v_n)$ in $(0,1]$
we consider the complemented subspace $Y_{\vv}=Y_{p,\vv}$ of $\ell_p$ with
corresponding projection $P_{\vv}=P_{p,\vv}$ as introduced in
Section~\ref{subsec:the-space-Yv}. Since $2<q$, the formal inclusion
$I_{Z_q,q}\colon Z_q=\big(\bigoplus_{n=1}^\infty
\ell_2^n\big)_{\ell_q} \to \ell_q=\big(\bigoplus _{n=1}^\infty
\ell_q^n\big)_{\ell_q}$ given by $I_{Z_q,q}\big( e^{(n)}_{2,j}\big) =
e^{(n)}_{q,j}$ is bounded, and hence, so is the formal inclusion
$I_{Y_{\vv},q}=I_{Z_q,q}\circ I_{Y_{\vv},Z_q}$. We shall consider the
closed ideal $\cJ^{I_{Y_{\vv},q}}$ which contains the operator
$S_{\vv}=I_{Y_{\vv},q} \circ P_{\vv}$.

As before, we first distinguish ideals corresponding to different
sequences.

\begin{thm}
  \label{thm:more-on-special-case}
  Let $\vv$ and $\vw$ be as in Theorem~\ref{thm:main}. Then
  $S_{\vw}\notin \cJ^{I_{Y_{\vv},q}}$.
\end{thm}
\begin{proof}
  We follow closely the proof of Theorem~\ref{thm:main}. For each
  $n\in\bn$ let $F_n=F^{(n)}_{p,v_n}$ and $G_n=F^{(n)}_{p,w_n}$ with
  bases $\big\{ f^{(n)}_j:\, 1\leq j\leq n\big\}$ and
  $\big\{ g^{(n)}_j:\, 1\leq j\leq n\big\}$, respectively. To simplify
  notation we write $Y=Y_{\vv}$ and $S=S_{\vw}$. Thus
  \[
  I_{Y,q}\colon Y= \left(\bigoplus _{n=1}^\infty F_n
  \right)_{\ell_p}\to \ell_q=\left( \bigoplus_{n=1}^\infty \ell_q^n
  \right)_{\ell_q}
  \]
  is given by $I_{Y,q}\big(f^{(n)}_j\big) = e^{(n)}_{q,j}$, and
  $S\colon \ell_p=\big(\bigoplus _{m=1}^\infty \ell_p^{k_m}\big)_{\ell_p}
  \to \ell_q$ is the composite $S=I_{Y_{\vw},q}\circ P_{\vw}$. Note
  that $S\big(g^{(m)}_i\big)=e^{(m)}_{q,i}$.

  For each $m\in\bn$ we define $\Phi_m\in
  \cL(\ell_p,\ell_q)^*$ by setting
  \[
  \Phi_m(V)= \frac1{m} \sum_{i=1}^m
  \bigip{e^{(m)}_{q',i}}{V(g^{(m)}_i)}\ ,\qquad
  V\in\cL(\ell_p,\ell_q)\ .
  \]
  Since $\norm{\Phi_m}\leq 1$ for all $m$, the sequence
  $(\Phi_m)$ has a $\omega^*$-accumulation point $\Phi$ in the unit
  ball of $\cL(\ell_p,\ell_q)^*$. Note that $\Phi_m(S)=1$ for all $m$, and
  hence $\Phi(S)=1$. The proof will be complete if we can show that
  $\cJ^{I_{Y,q}}$ is contained in the kernel of $\Phi$. To see this,
  fix $A\in \cL(\ell_q)$ and $B\in\cL(\ell_p,Y)$ with $\norm{A}\leq
  1$ and $\norm{B}\leq 1$. It is sufficient to show that
  \begin{equation}
    \label{eq:thm-more:aim}
    \Phi_m(AI_{Y,q}B)\to 0\qquad\text{as}\qquad m\to\infty\ .
  \end{equation}
  Let $B_m\colon G_m\to Y$ denote the restriction to $G_m$ of
  $B$. Exactly as in Theorem~\ref{thm:main} we use Lemma~\ref{lem:key}
  to show that
  \begin{equation}
    \label{eq:thm-more:aim2}
    \frac1{m}\sum_{i=1}^m \bignorm{B_m\big(g^{(m)}_i\big)}_{\infty}
    \to 0 \quad\text{as} \quad m\to\infty\ ,
  \end{equation}
  and obtain the estimate
  \begin{equation}
    \label{eq:more:bound-on-Phi-m}
    \abs{\Phi_m(AI_{Y,q}B)} \leq \frac1{m} \sum_{i=1}^m
    \bignorm{I_{Y,q}B_m\big(g^{(m)}_i\big)}_{\ell_q}\ .
  \end{equation}
  At this point we depart from the proof of Theorem~\ref{thm:main} as
  the argument becomes simpler. Let $x=\sum_{n=1}^\infty \sum_{j=1}^n
  x_{n,j} f^{(n)}_{j}\in Y$ with $\norm{x}_Y\leq 1$. Then for each
  $n\in\bn$ we have
  \begin{align*}
    \sum_{j=1}^n \abs{x_{n,j}}^q &= \sum_{j=1}^n
    \abs{x_{n,j}}^{q-2}\cdot \abs{x_{n,j}}^{2} \leq C_p^2\cdot
    \norm{x}_\infty^{q-2} \cdot \biggnorm{\sum_{j=1}^n
      x_{n,j}f^{(n)}_j}_{F_n}^2\\[2ex]
    &\leq C_p^2\cdot\norm{x}_\infty^{q-2} \cdot \biggnorm{\sum_{j=1}^n
      x_{n,j}f^{(n)}_j}_{F_n}^p\ . 
  \end{align*}
  It follows that
  \begin{align*}
    \norm{I_{Y,q}(x)}_{\ell_q}^q & = \sum_{n=1}^\infty \sum_{j=1}^n
    \abs{x_{n,j}}^q \leq C_p^2\cdot\norm{x}_\infty^{q-2} \cdot \sum_{n=1}^\infty
    \biggnorm{\sum_{j=1}^n x_{n,j}f^{(n)}_j}_{F_n}^p \\[2ex]
    &= C_p^2\cdot\norm{x}_\infty^{q-2} \cdot \norm{x}_Y^{p}\leq
    C_p^2\cdot\norm{x}_\infty^{q-2}\ .
  \end{align*}
  Applying this to~\eqref{eq:more:bound-on-Phi-m} we obtain
  \[
  \abs{\Phi_m(AI_{Y,q}B)} \leq C_p^{\frac2{q}}\cdot\frac1{m} \sum_{i=1}^m
  \bignorm{B_m\big(g^{(m)}_i\big)}_{\infty}^{1-\frac2{q}} \leq
  C_p^{\frac2{q}}\cdot \left(\frac1{m} \sum_{i=1}^m
  \bignorm{B_m\big(g^{(m)}_i\big)}_{\infty}\right)^{1-\frac2{q}}\ ,
  \]
  using the fact that the function $t\mapsto t^{1-\frac2{q}}$ is
  concave. Now we use~\eqref{eq:thm-more:aim2} to
  deduce~\eqref{eq:thm-more:aim}, as required.
\end{proof}

\begin{proof}[Proof of Theorem~\ref{mainthm:refinement}]
   The proof of Theorem~\ref{mainthm:infinite} provides a continuum
   size chain $\cC$ of infinite subsets of $\bn$, and corresponding
   sequences $\vw_M$, $M\in\cC$. In turn, this leads to spaces
   $Y_M=Y_{p,\vw_M}$ and closed ideals $\cJ^{I_{Y_M,q}}$. The proof of
   Theorem~\ref{mainthm:infinite} shows that if $M,N\in\cC$ and
   $M\subset N$, then $\vv=\vw_N$ and $\vw=\vw_M$ satisfy the
   hypotheses in Theorem~\ref{thm:main}. In particular, the unit
   vector basis of $Y_N$ dominates the unit vector basis of $Y_M$, and
   so $\cJ^{I_{Y_N,q}}\subset \cJ^{I_{Y_M,q}}$. Moreover, by
   Theorem~\ref{thm:more-on-special-case}, this inclusion is
   strict. Thus, the family $\big\{ \cJ^{I_{Y_M,q}}:\,M\in\cC\big\}$
   of closed ideals of $\cL(\ell_p.\ell_q)$ is a chain and has size
   the continuum.

   Finally, for each $M\in\cC$, the operator $I_{Y_M,q}$ factors
   through the formal inclusion $I_{2,q}\colon \ell_2=\big( \bigoplus
   _{n=1}^\infty \ell_2^n \big)_{\ell_2} \to \ell_q=\big( \bigoplus
   _{n=1}^\infty \ell_q^n \big)_{\ell_q}$ via the formal inclusion
   $I_{Y,2}$. Thus $\cJ^{I_{Y_M,q}}$ is contained in $\cJ^{I_{2,q}}$,
   and contains $\cJ^{I_{p,q}}$ since $I_{Y_M,q}$ is not compact.
\end{proof}

\section{Open problems}
\label{sec:problems}

There are a number of natural questions that remain or arise after our
work. The first aim would be to answer Pietsch's question in the range
$1\leq p<q$.

\begin{problem}
  Given $1\leq p<q<\infty$, are there infinitely many closed ideals in
  $\cL(\ell_1,\ell_q)$ or in $\cL(\ell_p,\co)$?
\end{problem}

Even in the reflexive range we do not know the exact number of closed
ideals. 

\begin{problem}
  Given $1<p<q<\infty$, find the cardinality of the set of closed
  ideals in $\cL(\ell_p,\ell_q)$.
\end{problem}
We now know this is at least $\continuum$. On the other hand, it is
clear that the cardinality of $\cL(\ell_p,\ell_q)$ is $\continuum$,
and hence there can be at most $2^\continuum$ closed ideals. Of
course, one could pose the same problem with $\ell_1$ replacing
$\ell_p$, etc.

Note that in the case $1<p<2<q<\infty$ we constructed two continuum
chains of closed ideals. Are these equal? More generally, we
could ask the following question about the lattice structure of closed
ideals in $\cL(\ell_p,\ell_q)$.

\begin{problem}
  Do the closed ideals of $\cL(\ell_p,\ell_q)$, after ignoring a
  finite number of them, form a chain?
\end{problem}

So far all our new ideals are generated by a single operator. Note
that if $T$ is a compact operator of infinite rank, then
$\cK=\cJ^T$. It is then natural to ask the following.

\begin{problem}
  Is $\fss$ generated by one operator? Are all closed ideals of
  $\cL(\ell_p,\ell_q)$ generated by one operator?
\end{problem}

Another candidate of a closed ideal, not representable by a single
operator could be the closure of the union of one of the chains we
defined in the previous section.

\end{document}